\documentclass[preprint,12pt]{elsarticle1}

\usepackage{subfig}
\usepackage{graphicx}
\usepackage{caption,color}   % 24.62
\usepackage{amssymb}
\usepackage{amsthm}
\usepackage{amsmath}
\usepackage{epic}
\usepackage{setspace}
\usepackage{float}
\usepackage{multirow}
\newtheorem{thm}{Theorem}[section]

\newtheorem{lem}[thm]{Lemma}

\newtheorem{remark}[thm]{Remark}

\numberwithin{equation}{section}
\journal{}

\begin{document}
\begin{spacing}{1.15}
\begin{frontmatter}
\title{The Laplacian spectral moments of power hypergraphs}

\author{Jueru Liu}
\author{Lixiang Chen}
\author{Changjiang Bu}\ead{buchangjiang@hrbeu.edu.cn}
\address{College of Mathematical Sciences, Harbin Engineering University, Harbin 150001, PR China}

\begin{abstract}
The $d$-th order Laplacian spectral moment of a $k$-uniform hypergraph is the sum of the $d$-th powers of all eigenvalues of its Laplacian tensor. In this paper, we obtain some expressions of the Laplacian spectral moments for $k$-uniform power hypergraphs, and these expressions can be represented by some parameters of graphs. And we show that some graphs can be determined by their high-order Laplacian spectrum by using the Laplacian spectral moments of power hypergraphs.
\end{abstract}

\end{frontmatter}

\section{Introduction}
For a $k$-uniform hypergraph $\mathcal{H}$, the $d$-th order (Laplacian) spectral moment of $\mathcal{H}$ is equal to the sum of the $d$-th powers of all eigenvalues of its adjacency (Laplacian) tensor. Since the $d$-th order trace of a tensor is equal to the sum of the $d$-th powers of its all eigenvalues \cite{ref7}, the $d$-th order (Laplacian) spectral moment of $\mathcal{H}$ is equal to the $d$-th order trace of its adjacency (Laplacian) tensor.

In 2013, Shao et al. \cite{ref8} gave a formula for the trace of tensors in terms of some graph parameters. The coefficients of characteristic polynomial and topological index of hypergraphs can be studied by spectral moments of hypergraphs \cite{ref9, ref10, ref11, ref22}. In 2021, Clark and Cooper \cite{ref10} expressed the spectral moments of hypergraph by the number of Veblen multi-hypergraphs and obtained the Harary-Sachs coefficient theorem for hypergraph. A formula for the spectral moment of a hypertree was given in terms of the number of some subhypertrees \cite{ref11}, and some high-order cospectral invariants of trees were given by the spectral moments of hypertrees \cite{ref16}. In \cite{ref22}, the Estrada index and subgraph centrality of uniform hypergraphs were studied, which are closely related to the traces of the adjacency tensor.

For Laplacian spectral moments of hypergraphs, the expressions of the first $k$ orders traces of the Laplacian tensors were given by the degree sequence of $k$-uniform hypergraphs \cite{ref9}. And an expression of the $k+1$-st order trace of Laplacian tensor of $k$-uniform hypergraphs was given in \cite{ref12}.

In this paper, we study Laplacian spectral moments of power hypergraphs. The expressions of the first $2k$ orders Laplacian spectral moments of $k$-uniform power hypergraphs are given, which can be represented by some parameters of graphs. And we show that some graphs, which are not determined by (signless) Laplacian spectrum, can be determined by their high-order (signless) Laplacian spectrum by considering the (signless) Laplacian spectral moments of power hypergraphs.

\section{Preliminaries}
Next, we introduce some notations and concepts for tensors and hypergraphs. For a positive integer $n$, let $[n]=\{1,2,\ldots,n\}$ and $[n]^k=\{i_1i_2\cdots i_k|i_j\in[n],j=1,\ldots,k\}$. A $k$-order $n$-dimension complex \textit{tensor} $\mathcal{T}=(t_{i\alpha})$ is a multi-dimensional array with $n^{k}$ entries on complex number field $\mathbb{C}$, where $i\in[n]$ and $\alpha\in[n]^{k-1}$.

A \textit{hypergraph} $\mathcal{H}=(V,E)$ consists of vertex set $V=\{1,2,\ldots,n\}$ and edge set $E=\{e_{1},e_{2},\ldots,e_{m}\}$, where $e_{j}\subseteq V(\mathcal{H})$ for $j\in[m]$. If $|e_{j}|=k$ for each $j\in[m]$ and $k\ge2$, then $\mathcal{H}$ is called a \textit{k-uniform} hypergraph. For a $k$-uniform hypergraph $\mathcal{H}$ with $n$ vertices, its \textit{adjacency tensor} $\mathcal{A}_{\mathcal{H}}=(a_{i\alpha})$ is a $k$-order $n$-dimension tensor has entries
		$$a_{i\alpha}=\left\{\begin{array}{ll}
			\frac{1}{(k-1)!},&  \textnormal{if}\ \{i,i_2,\ldots,i_k\}\in E(\mathcal{H})\ \textnormal{for}\ \alpha=i_2\cdots i_k,\\
			0,& \textnormal{otherwise}.
		\end{array}\right.$$
The spectrum and eigenvalues of $\mathcal{A}_{\mathcal{H}}$ is called the spectrum and eigenvalues of $\mathcal{H}$, repectively \cite{ref3}. For a vertex $i\in V(\mathcal{H})$, the \textit{degree} of $i$ is the number of edges of $\mathcal{H}$ containing the vertex $i$, denoted by $d_{i}$. The \textit{degree tensor} $\mathcal{D}_{\mathcal{H}}=\textnormal{diag}(d_1,\ldots,d_n)$ of $\mathcal{H}$ is a $k$-order $n$-dimension diagonal tensor. And tensor $\mathcal{L}_{\mathcal{H}}=\mathcal{D}_{\mathcal{H}}-\mathcal{A}_{\mathcal{H}}$ is the \textit{Laplacian tensor} of $\mathcal{H}$ \cite{ref5}.

In 2005, Lim \cite{ref1} and Qi \cite{ref2} introduced the eigenvalues of tensors independently. Denote the set of $n$-dimension complex vectors and the set of $k$-order $n$-dimension complex tensors by $\mathbb{C}^{n}$ and $\mathbb{C}^{[k,n]}$, respectively. For a tensor $\mathcal{T}=(t_{i\alpha})\in\mathbb{C}^{[k,n]}$ and $x=(x_{1},\ldots,x_{n})^{\mathsf{ T}}\in\mathbb{C}^{n}$, $\mathcal{T}x^{k-1}$ is a vector in $\mathbb{C}^{n}$ whose $i$-th component is $$(\mathcal{T}x^{k-1})_{i}=\sum\limits_{\alpha\in[n]^{k-1}}t_{i\alpha}x^{\alpha},$$ where $x^{\alpha}=x_{i_1}\cdots x_{i_{k-1}}$ if $\alpha=i_2\cdots i_{k-1}$. For a complex number $\lambda\in\mathbb{C}$, if there is a vector $x\in \mathbb{C}^n\setminus\{0\}$ such that $$\mathcal{T}x^{k-1}=\lambda x^{[k-1]},$$ then $\lambda$ is called an \textit{eigenvalue} of $\mathcal{T}$ and $x$ is an \textit{eigenvector} of $\mathcal{T}$ associated with $\lambda$, where $x^{[k-1]}=(x_{1}^{k-1},\ldots,x_{n}^{k-1})^{\mathsf{T}}$. The multi-set of all eigenvalues of tensor $\mathcal{T}$ is the \textit{spectrum} of $\mathcal{T}$, denoted by $\sigma(\mathcal{T})$.
	
In \cite{ref6}, an expression of $d$-th order trace for tensors is given. And Hu et al. \cite{ref7} proved that the $d$-th order trace of a $k$-order $n$-dimension tensor $\mathcal{T}$ is equal to the sum of the $d$-th powers of its all eigenvalues, that is, $\textnormal{Tr}_d(\mathcal{T})=\sum\nolimits_{\lambda\in\sigma(\mathcal{T})}\lambda^d$.

In 2013, Shao et al. \cite{ref8} gave a formula for $\textnormal{Tr}_d(\mathcal{T})$. Next, we introduce some related notations. For a positive integer $d$, let $$\mathcal{F}_{d}=\{(i_{1}\alpha_{1},\ldots,i_{d}\alpha_{d})|\ 1\le i_{1}\le\cdots\le i_{d}\le n;\alpha_{1},\ldots,\alpha_{d}\in[n]^{k-1}\}.$$ For $f=(i_{1}\alpha_{1},\ldots,i_{d}\alpha_{d})\in\mathcal{F}_{d}$ and a $k$-order $n$-dimension tensor $\mathcal{T}=(t_{i\alpha})$, let $\pi_{f}(\mathcal{T})=\prod\nolimits_{j=1}^{d} t_{i_j\alpha_j}$. Suppose $i_{j}\alpha_{j}=i_{j}v_{1}^{(j)}\cdots v_{k-1}^{(j)}$, let $E_{j}(f)=\{(i_{j},v_{1}^{(j)}),\ldots,(i_{j},v_{k-1}^{(j)})\}$ be the set of arcs from $i_{j}$ to $v_{1}^{(j)},\ldots,v_{k-1}^{(j)}$ and $E(f)=\bigcup\nolimits_{j=1}^{d} E_{j}(f)$ be an arc multi-set. Let $V_{j}(f)=\{i_{j},v_{1}^{(j)},\ldots,v_{k-1}^{(j)}\}$ and $V(f)=\bigcup\nolimits_{j=1}^{d} V_{j}(f)$ be a vertex set. Let multi-digraph $D(f)=(V(f),E(f))$. Let $b(f)$ be the product of the factorials of the multiplicities of all the arcs in $D(f)$. Let $c(f)$ be the product of the factorials of the outdegrees of all the vertices in $D(f)$. Let $W(f)$ be the set of all closed walks with the arc multi-set $E(f)$. In this paper, if a multi-set $\mathrm{A}$ contains $m$ distinct elements $a_{1},\ldots,a_{m}$ with multiplicities $r_{1},\ldots,r_{m}$ respectively, then we write $\mathrm{A}=\{a_{1}^{r_{1}},\ldots,a_{m}^{r_{m}}\}$.
	
The formula for the $d$-th order trace of tensors given by Shao et al. is shown as follows.
	
	\begin{lem}\textnormal{\cite{ref8}}
		\textnormal{Let $\mathcal{T}=(t_{i\alpha})$ be a $k$-order $n$-dimension tensor. Then}
		\begin{equation}
			\textnormal{Tr}_{d}(\mathcal{T})=(k-1)^{n-1} \sum_{f \in \mathcal{F}_{d}} \frac{b(f)}{c(f)} \pi_{f}(\mathcal{T})|W(f)|.
		\end{equation}
	\end{lem}

Since the $d$-th order Laplacian spectral moment of $\mathcal{H}$ is equal to the $d$-th order trace of its Laplacian tensor, we study the Laplacian spectral moment of uniform hypergraphs by considering the trace formula of tensor given by Shao et al.

For a $k$-uniform hypergraph $\mathcal{H}$ with $n$ vertices, let $\mathcal{L}_{\mathcal{H}}$ be the Laplacian tensor of $\mathcal{H}$. When $\mathcal{T}=\mathcal{L}_{\mathcal{H}}$ in Eq.(2.1), the $d$-th order Laplacian spectral moment of $\mathcal{H}$ is
	\begin{equation}
		 \textnormal{Tr}_{d}(\mathcal{L}_{\mathcal{H}})=(k-1)^{n-1}\sum_{f\in \mathcal{F}_{d}}\frac{b(f)}{c(f)}\pi_{f}(\mathcal{L}_{\mathcal{H}})|W(f)|.
	\end{equation}

Next, we simplify Eq.(2.2) by classifying $f$ and introduce some related concepts. For $i_{j}\alpha_{j}\in[n]^{k}$ and a $k$-order $n$-dimension tensor $\mathcal{T}=(t_{i\alpha})$, the entry $t_{i_{j}\alpha_{j}}$ in tensor $\mathcal{T}$ is called the corresponding entry of $i_{j}\alpha_{j}$. Suppose $\alpha_{j}=v_1^{(j)}\cdots v_{k-1}^{(j)}$, for a $k$-uniform hypergraph $\mathcal{H}$, $e=\{i_j,v_1^{(j)},\ldots,v_{k-1}^{(j)}\}$ is called the corresponding edge of tuple $i_{j}\alpha_{j}$ if the corresponding entry of $i_{j}\alpha_{j}$ in its adjacency tensor is not equal to zero.

Let $\pi_{f}(\mathcal{L}_{\mathcal{H}})|W(F)|\ne0$ for $f=(i_{1}\alpha_{1},\ldots,i_{d}\alpha_{d})\in\mathcal{F}_{d}$. Since $\pi_{f}(\mathcal{L}_{\mathcal{H}})=\prod\nolimits_{j=1}^{d}l_{i_{j}\alpha_{j}}\ne0$, we know $l_{i_{j}\alpha_{j}}\ne0$ for all $j\in[d]$. Then the tuple $i_{j}\alpha_{j}(j=1,\ldots,d)$ in $f$ corresponds either to a diagonal entry of $\mathcal{L}_{\mathcal{H}}$ or to an edge of $\mathcal{H}$. According to the number of the tuples which correspond to the diagonal entries of $\mathcal{L}_{\mathcal{H}}$, the set $\{f\in\mathcal{F}_{d}|\ \pi_{f}(\mathcal{L}_{\mathcal{H}})\ne0\}$ can be represented as the union of three disjoint sets, that is,
	\begin{equation}
		\{f\in\mathcal{F}_{d}|\ \pi_{f}(\mathcal{L}_{\mathcal{H}})\ne0\}=\mathcal{F}_{d}^{(1)}\cup\mathcal{F}_{d}^{(2)}\cup\mathcal{F}_{d}^{(3)},
	\end{equation}
	where $\mathcal{F}_{d}^{(1)}=\{f\in\mathcal{F}_{d}|\textnormal{ all tuples in}\ f\ \textnormal{correspond to diagonal entry of}\ \mathcal{L}_{\mathcal{H}}\}$, $\mathcal{F}_{d}^{(2)}=\{f=(i_{1}\alpha_{1},\ldots,i_{d}\alpha_{d})\in\mathcal{F}_{d}|\ \alpha_{j}=v_1^{(j)}\cdots v_{k-1}^{(j)}\ \textnormal{and}\ \{i_{j},v_1^{(j)},\ldots,v_{k-1}^{(j)}\}\in E(\mathcal{H})\ \textnormal{for}\ j=1,\ldots,d\}$, $\mathcal{F}_{d}^{(3)}=\{f\in\mathcal{F}_{d}|\ \pi_{f}(\mathcal{L}_{\mathcal{H}})\ne0\}\setminus(\mathcal{F}_{d}^{(1)}\cup\mathcal{F}_{d}^{(2)}).$

\begin{lem}
\textnormal{Let $\mathcal{H}$ be a $k$-uniform hypergraph with $n$ vertices. And the degree sequence of $\mathcal{H}$ is $d_1,d_2,\ldots,d_n$. Then}
\begin{equation}
(k-1)^{n-1}\sum\limits_{f\in\mathcal{F}_{d}^{(1)}}\frac{b(f)}{c(f)}\pi_{f}(\mathcal{L}_{\mathcal{H}})|W(f)|=(k-1)^{n-1}\sum\limits_{i=1}^{n}d_{i}^{d},
	\end{equation}
\begin{equation}
(k-1)^{n-1}\sum\limits_{f\in\mathcal{F}_{d}^{(2)}}\frac{b(f)}{c(f)}\pi_{f}(\mathcal{L}_{\mathcal{H}})|W(f)|=(-1)^{d}\textnormal{Tr}_{d}(\mathcal{A}_{\mathcal{H}}).
	\end{equation}
\end{lem}

\begin{proof}
For $f\in\mathcal{F}_{d}^{(1)}$, if $f=(i_{1}i_{1}\cdots i_{1},\ldots,i_{d}i_{d}\cdots i_{d})$, since the arc multi-set $E(f)$ only includes loops $(i_{j},i_{j})\ (j=1,\ldots,d)$, we know that $|W(f)|\ne0$ if and only if $i_{1}=\cdots=i_{d}$. Let $f_{i}=(ii\cdots i,\ldots,ii\cdots i)\in\mathcal{F}_{d}(i=1,\ldots,n)$, then $\mathcal{F}_{d}^{(1)}=\{f_{1},\ldots,f_{n}\}$. For $f_{i}\in\mathcal{F}_{d}^{(1)}$, since $b(f_{i})=c(f_{i})=(d(k-1))!$, $|W(f_{i})|=1$ and $\pi_{f_{i}}(\mathcal{L}_{\mathcal{H}})=l_{ii\cdots i}^{d}=d_{i}^{d}$, Eq.(2.4) can be obtained directly.
	
For $f=(i_{1}\alpha_{1},\ldots,i_{d}\alpha_{d})\in\mathcal{F}_{d}^{(2)}$, where $\alpha_{j}=v_1^{(j)}\cdots v_{k-1}^{(j)}$ for $j=1,\ldots,d$. Since $\{i_{j},v_1^{(j)},\ldots,v_{k-1}^{(j)}\}\in E(\mathcal{H})$ for $j=1,\ldots,d$, we have $\pi_{f}(\mathcal{L}_{\mathcal{H}})=\prod\nolimits_{j=1}^{d}l_{i_{j}\alpha_{j}}=(-\frac{1}{(k-1)!})^{d}=(-1)^{d}\pi_{f}(\mathcal{A}_{\mathcal{H}}).$ And $\pi_{f}(\mathcal{A}_{\mathcal{H}})\ne0$ if and only if $\{i_{j},v_1^{(j)},\ldots,v_{k-1}^{(j)}\}\in E(\mathcal{H})$ for $j=1,\ldots,d$, that is, $f\in\mathcal{F}_{d}^{(2)}$, then Eq.(2.5) can be obtained.
\end{proof}	

According to Lemma 2.2, in order to obtain the expressions of the first $2k$ orders Laplacian spectral moments for $k$-uniform power hypergraphs, we should give some expressions of the spectral moments for $k$-uniform power hypergraphs.

For a graph $G$ and a positive integer $k\ge3$, the \textit{$k$-power hypergraph} of $G$, denoted by $G^{(k)}$, is a $k$-uniform hypergraph obtained by adding $k-2$ new vertices whose degrees are $1$ to each edge of $G$ \cite{ref4}. The spectrum of a hypergraph is said to be \textit{$k$-symmetric}, if its spectrum is invariant under a rotation of an angle $2\pi/k$ in the complex plane. Shao et al. \cite{ref8} gave a characterization (in terms of the traces of the adjacency tensors) of the $k$-uniform hypergraphs whose spectrum are $k$-symmetric, that is, the spectrum of a $k$-uniform hypergraph $\mathcal{H}$ is $k$-symmetric if and only if $\textnormal{Tr}_{d}(\mathcal{A}_{\mathcal{H}})=0$ for $k\nmid d$. It is obvious that the spectrum of a $k$-uniform power hypergraph is $k$-symmetric. Then, the $d$-th spectral moments of $G^{(k)}$ are equal to $0$ for $d=k+1,\ldots,2k-1$, that is,
\begin{equation}
\textnormal{Tr}_d(\mathcal{A}_{G^{(k)}})=0\ \mathrm{for}\ d=k+1,\ldots,2k-1.
\end{equation}

And the expression of the $2k$-th order spectral moment of $G^{(k)}$ is given as follows.
	
	\begin{lem}
		\textnormal{Let $G$ be a graph with $n$ vertices and $m$ edges. Let $d_i$ denote the degree of vertex $i$ in $G$ ($i=1,\ldots,n$). Then the $2k$-th order spectral moment of $G^{(k)}$ is}
        \begin{small}
		\begin{equation}
			 \textnormal{Tr}_{2k}(\mathcal{A}_{G^{(k)}})=k^{k-1}(k-1)^{N-k}\big(1-2k^{k-3}(k-1)^{1-k}\big)m+k^{2k-3}(k-1)^{N-2k+1}\sum\limits_{i=1}^{n}d_i^2,
		\end{equation}
        \end{small}
		\textnormal{where $N=n+m(k-2)$.}
	\end{lem}

	\begin{proof}
		Let $\mathcal{G}=G^{(k)}$. Then $|V(\mathcal{G})|=n+m(k-2)=N$ and $|E(\mathcal{G})|=m$. Let $N_G(P_2)$ and $N_{\mathcal{G}}(P_2^{(k)})$ denote the number of paths with length $2$ in $G$ and $\mathcal{G}$, respectively. Then $N_{\mathcal{G}}(P_2^{(k)})=N_G(P_2)=\sum\limits_{i=1}^{n}\dbinom{d_i}{2}$. From Lemma 2.1, we get $$\textnormal{Tr}_{2k}(\mathcal{A}_{\mathcal{G}})=(k-1)^{N-1}\sum\limits_{f\in\mathcal{F}_{2k}}\frac{b(f)}{c(f)}\pi_{f}(\mathcal{A}_{\mathcal{G}})|W(f)|.$$

For $f=(i_{1}\alpha_{1},\ldots,i_{2k}\alpha_{2k})\in\mathcal{F}_{2k}$, if $\pi_{f}(\mathcal{A}_{\mathcal{G}})=\prod_{j=1}^{2k}a_{i_j\alpha_j}\ne0$, then $a_{i_j\alpha_j}\ne0$ for all $j\in[2k]$. For $|W(f)|\ne0$, there are the following two cases.

Case 1: $f=(i_{1}\alpha_{1},i_{1}\beta_{1},\ldots,i_{k}\alpha_{k},i_{k}\beta_{k})=f_{e}\in\mathcal{F}_{2k}$, where $\{i_{1},\ldots, i_{k}\}=e\in E(\mathcal{G})$ and $\alpha_{j},\beta_{j}\in\big(\{i_1,\ldots,i_k\}\setminus \{i_j\}\big)^{k-1}$ for $j=1,\ldots,k$. Then
\begin{align*}
&\sum_{e\in E(\mathcal{G})}\frac{b(f_e)}{c(f_e)}\pi_{f_e}(\mathcal{A}_{\mathcal{G}})|W(f_e)|\\
=&\sum_{e\in E(\mathcal{G})}\frac{(2!)^{k(k-1)}}{\big((2k-2)!\big)^k}\Big(\frac{1}{(k-1)!}\Big)^{2k}\frac{2k(k-1)}{(2!)^{k(k-1)}}2^{k-1}k^{k-2}\big((2k-3)!\big)^k\big((k-1)!\big)^{2k}\\
=&k^{k-1}(k-1)^{1-k}|E(\mathcal{G})|.
\end{align*}

Case 2: $f=(i_{1}\alpha_{1},j_{1}\beta_{1},i_{2}\alpha_{2},\ldots,i_{k}\alpha_{k},j_{2}\beta_{2},\ldots,j_{k}\beta_{k})=f_{e_1 e_2}\in\mathcal{F}_{2k}$, where $i_1=j_1$, $\{i_{1},i_{2},\ldots,i_{k}\}=e_1\in E(\mathcal{G}),\ \{i_{1},j_{2},\ldots,j_{k}\}=e_2\in E(\mathcal{G})$ and $\alpha_{l}\in\big(\{i_1,\ldots,i_k\}\setminus \{i_l\}\big)^{k-1},\ \beta_{l}\in\big(\{j_1,\ldots,j_k\}\setminus \{j_l\}\big)^{k-1}$ for $l=1,\ldots,k$. Then
\begin{align*}
&\sum_{e_1 e_2\subset \mathcal{G}}\frac{b(f_{e_1 e_2})}{c(f_{e_1 e_2})}\pi_{f_{e_1 e_2}}(\mathcal{A}_{\mathcal{G}})|W(f_{e_1 e_2})|\\
=&\sum_{e_1 e_2\subset \mathcal{G}}\frac{2k(k-1)(k^{k-2})^2(2k-3)!\big((k-2)!\big)^{2k-2}}{\big(2(k-1)\big)!\big((k-1)!\big)^{2k-2}}\Big(\frac{1}{(k-1)!}\Big)^{2k}2\big((k-1)!\big)^{2k}\\
=&2k^{2k-3}(k-1)^{2-2k}N_{\mathcal{G}}(P_2^{(k)}).
\end{align*}

		Then
		\begin{align*} \textnormal{Tr}_{2k}(\mathcal{A}_{\mathcal{G}})
&=(k-1)^{N-1}\Big(k^{k-1}(k-1)^{1-k}|E(\mathcal{G})|+2k^{2k-3}(k-1)^{2-2k}N_{\mathcal{G}}(P_2^{(k)})\Big)\\
&=k^{k-1}(k-1)^{N-k}\big(1-2k^{k-3}(k-1)^{1-k}\big)m+k^{2k-3}(k-1)^{N-2k+1}\sum\limits_{i=1}^{n}d_i^2,
		\end{align*}
where $N=n+m(k-2)$.
	\end{proof}
	
	\section{Main results}
In this section, we give an expression of the $d$-th order Laplacian spectral moments for $k$-uniform hypergraphs. And the explicit expressions of the first $2k$ orders Laplacian spectral moments for $k$-uniform power hypergraphs are given.

Given two hypergraphs $\mathcal{H}=(V(\mathcal{H}),E(\mathcal{H}))$ and $H=(V(H),E(H))$, if $V(H)\subseteq V(\mathcal{H})$ and $E(H)\subseteq E(\mathcal{H})$, then $H$ is said to be a \textit{subhypergraph} of $\mathcal{H}$. A $k$-uniform \textit{multi-hypergraph} $\mathcal{H}$ is a pair $(V(\mathcal{H}),E(\mathcal{H}))$, where $E(\mathcal{H})$ is a multi-set of subsets of $V(\mathcal{H})$ with cardinality $k$. A \textit{Veblen hypergraph} is a $k$-uniform, $k$-valent (i.e., the degree of every vertex is a multiple of $k$) multi-hypergraph \cite{ref10}. For a multi-hypergraph $H$, let $\underline{H}$ be the simple $k$-uniform hypergraph formed by removing duplicate edges of $H$. And $H$ is called a \textit{multi-subgraph} of $\mathcal{H}$ if $\underline{H}$ is a subhypergraph of $\mathcal{H}$. Let $\mathcal{V}_{d}(\mathcal{H})$ denote the set of connected Veblen multi-subgraph of $\mathcal{H}$ with $d$ edges.

For $f=(i_{1}\alpha_{1},\ldots,i_{d}\alpha_{d})\in\mathcal{F}_{d}$ and a $k$-uniform hypergraph $\mathcal{H}$ (where $\alpha_{j}=v_1^{(j)}\cdots v_{k-1}^{(j)}$ for $j=1,\ldots,d$), the \textit{multi-subgraph induced by $f$}, denoted by $H(f)$, is the multi-hypergraph with the vertex set $V(f)\subseteq V(\mathcal{H})$ and the edge multi-set $E(H(f))=\{\{i_j,v_1^{(j)},\ldots,v_{k-1}^{(j)}\}|\ (\mathcal{A}_{\mathcal{H}})_{i_{j}\alpha_{j}}\ne0,\ 1\le j\le d\}$, and $\underline{H}(f)$ is a subhypergraph of $\mathcal{H}$.
	
A \textit{walk} ia a digraph $D$ is a non-empty alternating sequence $v_0e_0v_1e_1\cdots v_ke_k$ of vertices and arcs in $D$ such that $e_i=(v_i,v_{i+1})$ for all $i<k$. A walk is \textit{closed} if $v_0=v_k$. A closed walk in a digraph is an \textit{Eulerian closed walk} if it traverses each arc of this digraph exactly once. A digraph $D$ is called \textit{Eulerian} if $D$ has an Eulerian closed walk. Let $d^{+}(v)$ and $d^{-}(v)$ be the outdegree and indegree of the vertex $v\in V(D)$, respectively. The digraph $D$ is Eulerian if and only if $d^{+}(v)=d^{-}(v)$ for all $v\in V(D)$ and $D$ is weakly connected.
	
	Since $W(f)$ is the set of all closed walks with the arc multi-set $E(f)$, we know that $|W(f)|$ is equal to the number of Eulerian closed walks in the multi-digraph $D(f)$. For $f\in\mathcal{F}_{d}^{(3)}$, we give the following conclusion.
	
	\begin{lem}
		\textnormal{Let $\mathcal{H}$ be a $k$-uniform hypergraph with $n$ vertices. If $f\in\mathcal{F}_{d}^{(3)}$ and $|W(f)|\ne0$, then the multi-subgraph $H(f)$ induced by $f$ is a connected Veblen multi-subgraph of $\mathcal{H}$ with at most $d-1$ edges.}
	\end{lem}
	
	\begin{proof}
		For $f=(i_{1}\alpha_{1},\ldots,i_{d}\alpha_{d})\in\mathcal{F}_{d}$ (where $\alpha_{j}=v_1^{(j)}\cdots v_{k-1}^{(j)}$ for $j=1,\ldots,d$), if $|W(f)|\ne0$, then the multi-digraph $D(f)=(V(f),E(f))$ is Eulerian. For all $v\in V(f)$, we have $d^{+}_{D(f)}(v)=d^{-}_{D(f)}(v)$ and
		\begin{align*}
			d^{+}_{D(f)}(v)&=(k-1)|\{i_{j}\alpha_{j}|\ i_{j}=v\}|\\&=(k-1)\big(|\{i_{j}\alpha_{j}|\ i_{j}=v\ \textnormal{and}\ \{i_j,v_1^{(j)}\cdots v_{k-1}^{(j)}\}\in E(\mathcal{H})\}|+|\{i_{j}\alpha_{j}|\ i_{j}\alpha_{j}=vv\cdots v\}|\big),\\
			d^{-}_{D(f)}(v)&=|\{i_{j}\alpha_{j}|\ i_{j}\ne v\ \textnormal{and}\ v\in V_{j}(f)\}|+(k-1)|\{i_{j}\alpha_{j}|\ i_{i}\alpha_{j}=vv\cdots v\}|.
		\end{align*}
		Then
		$$(k-1)|\{i_{j}\alpha_{j}|\ i_{j}=v\ \textnormal{and}\ \{i_j,v_1^{(j)}\cdots v_{k-1}^{(j)}\}\in E(\mathcal{H})\}|=|\{i_{j}\alpha_{j}|\ i_{j}\ne v\ \textnormal{and}\ v\in V_{j}(f)\}|.$$
		Fix a vertex $v \in V(H(f))$. We have
		\begin{align*}
			d_{H(f)}(v)&=|\{i_{j}\alpha_{j}|\ i_{j}=v\ \textnormal{and}\ \{i_j,v_1^{(j)}\cdots v_{k-1}^{(j)}\}\in E(\mathcal{H})\}|+|\{i_{j}\alpha_{j}|\ i_{j}\ne v\ \textnormal{and}\ v\in V_{j}(f)\}|\\&=k|\{i_{j}\alpha_{j}|\ i_{j}=v\ \textnormal{and}\ \{i_j,v_1^{(j)}\cdots v_{k-1}^{(j)}\}\in E(\mathcal{H})\}|.
		\end{align*}
		So $k | d_{H(f)}(v)$, it follows that $H(f)$ is a Veblen hypergraph by definition. And $f\in\mathcal{F}_{d}^{(3)}$, then $H(f)$ has at most $d-1$ edges.
	\end{proof}

For a connected Veblen multi-subgraph $H$ of $\mathcal{H}$ and $f=(i_{1}\alpha_{1},\ldots,i_{d}\alpha_{d})\in\mathcal{F}_d$, $f$ is called corresponding to $H$ if $f$ satisfy the following conditions:

    {\noindent}(a) there is a integer $l(1\le l\le d-1)$, such that $i_{j_1}\alpha_{j_1},\ldots,i_{j_l}\alpha_{j_l}$ are corresponding to some edges of $H$;

    {\noindent}(b) for every edge $e\in E(H)$, there exists $j\in[d]$ such that $i_j\alpha_j$ is corresponding to $e$;

    {\noindent}(c) and others in $f$ are $v\beta_{v}$ where $\beta_{v}=v\cdots v\in[n]^{k-1}$ for $v\in V(H)$.

    Let $\mathcal{F}_{d}(H)=\{f\in\mathcal{F}_d|\ f\ \textnormal{is corresponding to}\ H\}$. From Lemma 3.1, we have $$\{f\in\mathcal{F}_d^{(3)}| |W(f)|\ne0\}=\bigcup\limits_{z=1}^{d-1}\bigcup\limits_{H\in\mathcal{V}_z(\mathcal{H})}\mathcal{F}_d(H).$$
	
	For simplicity, $\tau(D(f))$ is abbreviated to $\tau(f)$, which is the number of arborescences in multi-digraph $D(f)$. According to the above process, the formula for the $d$-th order Laplacian spectral moment of $k$-uniform hypergraphs is given as follows.
	
	\begin{thm}
		\textnormal{Let $\mathcal{H}$ be a $k$-uniform hypergraph with $n$ vertices. And the degree sequence of $\mathcal{H}$ is $d_1,d_2,\ldots,d_n$. Then}
		\begin{align*} &\textnormal{Tr}_{d}(\mathcal{L}_{\mathcal{H}})=(k-1)^{n-1}\sum\limits_{i=1}^{n}d_{i}^{d}+(-1)^{d}\textnormal{Tr}_{d}(\mathcal{A}_{\mathcal{H}})+d(k-1)^{n}\sum\limits_{z=1}^{d-1}\sum\limits_{H\in\mathcal{V}_{z}(\mathcal{H})}\sum\limits_{f\in\mathcal{F}_{d}(H)}\frac{\tau(f)\pi_{f}(\mathcal{L}_{\mathcal{H}})}{\prod\limits_{v\in V(f)}d^{+}(v)}.
		\end{align*}
	\end{thm}

\begin{proof}
From Eq.(2.3), the $d$-th Laplacian spectral moments of $\mathcal{H}$ is
\begin{equation}
\textnormal{Tr}_{d}(\mathcal{L}_{\mathcal{H}})=(k-1)^{n-1}\sum\limits_{j=1}^{3}\sum\limits_{f\in\mathcal{F}_{d}^{(j)}}\frac{b(f)}{c(f)}\pi_{f}(\mathcal{L}_{\mathcal{H}})|W(f)|.
\end{equation}

For $f\in\mathcal{F}_{d}^{(3)}$, let $\widetilde{D}(f)$ be the digraph obtained by removing all repeated arcs of $D(f)$. Then $c(f)=\prod\nolimits_{v\in V(f)}d^{+}(v)!$, $b(f)=\prod_{e\in\widetilde{D}(f)}w(e)!$ and $|E(f)|=d(k-1)$. From Theorem 6 in \cite{ref13}, the number of Eulerian closed walks in $D(f)$ is
\begin{equation}
|W(f)|=\frac{|E(f)|}{b(f)}|\mathfrak{E}(f)|,
\end{equation}
where $|\mathfrak{E}(f)|$ is the number of the Eulerian circuits in $D(f)$.

From BEST Theorem \cite{ref14,ref15}, the number of the Eulerian circuits in $D(f)$ is
\begin{equation}
|\mathfrak{E}(f)|=\tau(f)\prod\limits_{v\in V(f)}(d^{+}(v)-1)!.
\end{equation}

According to Eq.(3.2) and Eq.(3.3), we have
\begin{equation}
\begin{aligned}
&(k-1)^{n-1}\sum\limits_{f\in\mathcal{F}_{d}^{(3)}}\frac{b(f)}{c(f)}\pi_{f}(\mathcal{L}_{\mathcal{H}})|W(f)|\\
=&(k-1)^{n-1}\sum\limits_{z=1}^{d-1}\sum\limits_{H\in\mathcal{V}_{z}(\mathcal{H})}\sum\limits_{f\in\mathcal{F}_{d}(H)}\frac{b(f)}{c(f)}\pi_{f}(\mathcal{L}_{\mathcal{H}})|W(f)|\\
=&d(k-1)^n\sum\limits_{z=1}^{d-1}\sum\limits_{H\in\mathcal{V}_{z}(\mathcal{H})}\sum\limits_{f\in\mathcal{F}_{d}(H)}\frac{\tau(f)}{\prod\limits_{v\in V(f)}d^{+}(v)}\pi_{f}(\mathcal{L}_{\mathcal{H}}).
\end{aligned}
\end{equation}

Then we can obtain the expression for the $d$-th order Laplacian spectral moment of $\mathcal{H}$ by substituting Eq.(2.4), Eq.(2.5) and Eq.(3.4) into Eq.(3.1).
\end{proof}
	
\begin{remark}
\textnormal{Since $\textnormal{Tr}_{d}(\mathcal{A}_{\mathcal{H}})=0\ (d=1,\ldots,k-1)$ and $\textnormal{Tr}_{k}(\mathcal{A}_{\mathcal{H}})=k^{k-1}(k-1)^{n-k}|E(\mathcal{H})|$ \cite{ref3}, $\mathcal{V}_d(\mathcal{H})=\emptyset$ for $d=1,\ldots,k-1$. The expressions of the first $k$ orders Laplacian spectral moments of a $k$-uniform hypergraph $\mathcal{H}$ can be obtained directly by using the formulas given in Theorem 3.1. And these expressions have been given in \cite{ref9}.}
\end{remark}

\begin{remark}
\textnormal{Since $\textnormal{Tr}_{k+1}(\mathcal{A}_{\mathcal{H}})=(k+1)(k-1)^{n-k}C_k(\#\ \textnormal{of simplices in}\ \mathcal{H})$ \cite{ref3}, and $\mathcal{V}_k(\mathcal{H})=\{ke|\ e\in E(\mathcal{H})\}$. The expressions of the $k+1$-st orders Laplacian spectral moments of a $k$-uniform hypergraph $\mathcal{H}$ can be obtained by using the formulas given in Theorem 3.1. And these expressions have been given in \cite{ref12}.}
\end{remark}

If the $k$-uniform hypergraph is the $k$-power hypergraph $G^{(k)}$ of a graph $G$ in Remark 3.1 and Remark 3.2, the expressions of the first $k+1$ order Laplacian spectral moments for $G^{(k)}$ can be obtained.
	
For a graph $G$, the expressions of the first $2k$ orders Laplacian spectral moments of its $k$-power hypergraph $G^{(k)}$ can be given by considering the formulas shown in Theorem 3.1, and these expressions are represented by some parameters of $G$.
	
	\begin{thm}
		\textnormal{Let $G$ be a graph with $n$ vertices and $m$ edges. Let $d_i$ denote the degree of vertex $i$ in $G$ ($i=1,\ldots,n$). For the $k$-power hypergraph $G^{(k)}$ of $G$, then}
\begin{align*}
\textnormal{Tr}_d(\mathcal{L}_{G^{(k)}})&=(k-1)^{N-1}\sum\limits_{i=1}^nd_i^d+(-1)^kdk^{k-2}(k-1)^{N-k}\Big(\sum\limits_{i-1}^nd_i^{d-k+1}+\sum\limits_{\{i,j\}\in E(G)}N_{d-k}(d_i,d_j)\Big)\\
&+(k-1)^{N-k}\big((k-1)^{k-1}+(-1)^kdk\big)(k-2)m,
\end{align*}
\textnormal{for $d=k+1,\ldots,2k-1$, and}
\begin{align*}
\textnormal{Tr}_{2k}(\mathcal{L}_{G^{(k)}})&=(k-1)^{N-1}\sum\limits_{i=1}^nd_i^{2k}+(-1)^k2k^{k-1}(k-1)^{N-k}\Big(\sum\limits_{i-1}^nd_i^{k+1}+\sum\limits_{\{i,j\}\in E(G)}N_{k}(d_i,d_j)\Big)\\
&+k^{2k-3}(k-1)^{N-2k+1}\sum\limits_{i=1}^nd_i^2+\ell m,
\end{align*}
		\textnormal{where} $N=n+m(k-2)$, $N_{s}(d_i,d_j)=\sum\nolimits_{\begin{subarray}{c} 1\le c_i+c_j\le s \\ 0\le c_i,c_j<s \end{subarray}}d_i^{c_i}d_j^{c_j}\big(s=1,\ldots,k\big)$, $\ell=(k-1)^{N-k}\big((k-1)^{k-1}(k-2)+(-1)^k2k^{k-1}(k-2)+k^{k-1}-2k^{2k-3}(k-1)^{1-k}\big)$.
	\end{thm}
	
	\begin{proof}
Let $\mathcal{G}=G^{(k)}$, then $|V(\mathcal{G})|=N=n+m(k-2)$ and $|E(\mathcal{G})|=m$.

Since $\mathcal{G}$ is the $k$-power hypergraph of $G$, from the definition of Veblen hypergraph, we know that $\mathcal{V}_{d}(\mathcal{G})=\emptyset$ for $k\nmid d$ and $d\in[2k]$. For a Veblen multi-subgraph $H\in\mathcal{V}_{k}(\mathcal{G})$, it is easy to see that $\underline{H}\in E(\mathcal{G})$. For convenience, let $ke$ denote the connected Veblen multi-subgraph with $k$ edges and $\underline{ke}=e\in E(\mathcal{G})$. Then, $\mathcal{V}_{k}(\mathcal{G})=\{ke|\ e\in E(\mathcal{G})\}$. For $d=k+1,\ldots,2k$, we have

\begin{align*}
\textnormal{Tr}_d(\mathcal{L}_{\mathcal{G}})=&(k-1)^{N-1}\Big(m(k-2)+\sum\limits_{i=1}^{n}d_{i}^{d}\Big)+(-1)^{d}\textnormal{Tr}_{d}(\mathcal{A}_{\mathcal{G}})\\&+d(k-1)^{N}\sum\limits_{e\in E(\mathcal{G})}\sum\limits_{f\in\mathcal{F}_{d}(ke)}\frac{\tau(f)}{\prod\limits_{v\in V(f)}d^{+}(v)}\pi_{f}(\mathcal{L}_{\mathcal{G}}).
\end{align*}

For $f\in\mathcal{F}_{d}(ke)\ (\textnormal{where}\ e=\{i_{1},i_{2},\ldots,i_{k}\}\in E(\mathcal{G}))$, let
$$f=((i_1 \beta_1)^{c_1}, i_1\alpha_1, (i_2 \beta_2)^{c_2}, i_2\alpha_2, \ldots,(i_k \beta_k)^{c_k}, i_k\alpha_k),$$
where $i_{1}<i_{2}<\cdots<i_{k}$. For any $j\in[k]$, $\alpha_{j}\in(\{i_1,\ldots,i_k\}\setminus\{i_j\})^{k-1}$, $\beta_{j}=i_{j}\cdots i_{j}\in[N]^{k-1}$, $c_{j}\ge0$ is the total number of times that $i_{j}\beta_{j}$ appears in $f$, and $\sum\nolimits_{j=1}^{k}c_j=d-k$. Next, we consider the following two cases for $f\in\mathcal{F}_{d}(ke)$.
	
Case 1: If there exists $j\in[k]$ such that $c_{j}=d-k$, then
$$f=f_{e,i_{j}}=(i_1\alpha_1, \ldots, i_{j-1}\alpha_{j-1}, (i_j \beta_j)^{d-k}, i_j\alpha_j, \ldots, i_k\alpha_k)\in\mathcal{F}_{d}(ke).$$
We have
$$\tau(f)=k^{k-2},\ \prod\limits_{v\in V(f)}d^{+}(v)=(d-k+1)(k-1)^{k},$$
and there are $(d-k+1)((k-1)!)^{k}$ elements in $\mathcal{F}_{d}$ which share the same arc multi-set as $f_{e,i_{j}}$, then
$$\sum\limits_{j=1}^{k}\frac{\tau(f)}{\prod\limits_{v\in V(f)}d^{+}(v)}\pi_{f}(\mathcal{L}_{\mathcal{G}})=(-1)^{k}k^{k-2}(k-1)^{-k}\sum\limits_{j=1}^{k} d_{i_j}^{d-k}.$$
		
Case 2: If for any $j\in[k]$, $0\le c_{j}<d-k$, then
$$f=f_{e,\{c_1,c_2,\ldots,c_k\}}=((i_1 \beta_1)^{c_1}, i_1\alpha_1, (i_2 \beta_2)^{c_2}, i_2\alpha_2, \ldots,(i_k \beta_k)^{c_k}, i_k\alpha_k)\in\mathcal{F}_{d}(ke).$$
We have
$$\tau(f)=k^{k-2},\ \prod\limits_{v\in V(f)}d^{+}(v)=(k-1)^{k}\prod\limits_{j=1}^{k}(c_{j}+1),$$
and there are $((k-1)!)^{k}\prod\limits_{j=1}^{k}(c_{j}+1)$ elements in $\mathcal{F}_{d}$ which share the same arc multi-set as $f_{e,\{c_1,c_2,\ldots,c_k\}}$, then
$$\sum\limits_{\begin{subarray}{c} c_1+\cdots+c_k=d-k \\ \forall j\in[k],0\le c_{j}<d-k \end{subarray}}\frac{\tau(f)}{\prod\limits_{v\in V(f)}d^{+}(v)}\pi_{f}(\mathcal{L}_{\mathcal{G}})=(-1)^{k}k^{k-2}(k-1)^{-k}\sum\limits_{\begin{subarray}{c} c_1+\cdots+c_k=d-k \\ \forall j\in[k],0\le c_{j}<d-k \end{subarray}}\prod\limits_{j=1}^{k}d_{i_j}^{c_j}.$$
		
		{\noindent}Then
		\begin{small}
		\begin{align*}
			&\sum\limits_{e\in E(\mathcal{G})}\sum\limits_{f\in\mathcal{F}_{d}(ke)}\frac{\tau(f)}{\prod\limits_{v\in V(f)}d^{+}(v)}\pi_{f}(\mathcal{L}_{\mathcal{G}})\\
=&(-1)^{k}k^{k-2}(k-1)^{-k}\sum\limits_{\{i_{1},\ldots,i_{k}\}\in E(\mathcal{G})}\bigg(\sum\limits_{j=1}^{k} d_{i_j}^{d-k}+\sum\limits_{\begin{subarray}{c} c_1+\cdots+c_k=d-k \\ \forall j\in[k],0\le c_{j}<d-k \end{subarray}}\prod\limits_{j=1}^{k}d_{i_j}^{c_j}\bigg)\\
=&(-1)^{k}k^{k-2}(k-1)^{-k}\bigg(\sum\limits_{i=1}^{N}\sum\limits_{e\in E_{i}} d_{i}^{d-k}+\sum\limits_{\{i_{1},\ldots,i_{k}\}\in E(\mathcal{G})}\sum\limits_{\begin{subarray}{c} c_1+\cdots+c_k=d-k \\ \forall j\in[k],0\le c_{j}<d-k \end{subarray}}\prod\limits_{j=1}^{k}d_{i_j}^{c_j}\bigg)\\
=&(-1)^{k}k^{k-2}(k-1)^{-k}\bigg(\sum\limits_{i=1}^{N}d_{i}^{d-k+1}+\sum\limits_{\{i_{1},\ldots,i_{k}\}\in E(\mathcal{G})}\sum\limits_{\begin{subarray}{c} c_1+\cdots+c_k=d-k \\ \forall j\in[k],0\le c_{j}<d-k \end{subarray}}\prod\limits_{j=1}^{k}d_{i_j}^{c_j}\bigg).
		\end{align*}
	    \end{small}

For $e=\{i,j\}\in E(G)$, let $e^{(k)}=\{i,j\}^{(k)}=\{i,j,v_{e,1},\ldots,v_{e,k-2}\}\in E(G^{(k)})$, where $v_{e,l}$ are cored vertex (the vertex whose degree is 1 \cite{ref4}), then $d_{e,l}=1(l=1,\ldots,k-2)$ and $1\le c_i+c_j=d-k-\sum\nolimits_{l=1}^{k-2}c_{e,l}\le d-k$. Then, for $d=k+1,\ldots,2k$, the $d$-th order Laplacian spectral moment of $G^{(k)}$ is
\begin{align*}
\textnormal{Tr}_d(\mathcal{L}_{G^{(k)}})=&(-1)^{d}\textnormal{Tr}_d(\mathcal{A}_{G^{(k)}})+(k-1)^{N-1}\Big(m(k-2)+\sum\limits_{i=1}^{n}d_{i}^{d}\Big)\\+&(-1)^{k}dk^{k-2}(k-1)^{N-k}\bigg(\sum\limits_{i=1}^{N}d_{i}^{d-k+1}+\sum\limits_{\{i,j\}\in E(G)}\sum\limits_{\begin{subarray}{c} 1\le c_i+c_j\le d-k \\ 0\le c_i,c_j<d-k \end{subarray}}d_{i}^{c_i}d_{j}^{c_j}\bigg).
\end{align*}

By substituting Eq.(2.6) and Eq.(2.7) into the above equation, the expressions of $\textnormal{Tr}_d(\mathcal{L}_{G^{(k)}})(d=k+1,\ldots,2k)$ can be obtained. 	
	\end{proof}

Let $\sum\nolimits_{i=s}^{t}a_i=0$ if $t<s$. Let $G=(V(G),E(G))$ be a finite simple graph. Let $d_{v}$ denote the degree of the vertex $v$ in $G$. The first and second Zagreb indices were introduced in \cite{ref17,ref18}, which are $M_1(G)=\sum\nolimits_{v\in V(G)}d_v^2=\sum\nolimits_{\{u,v\}\in E(G)}\big(d_u+d_v\big)$ and $M_2(G)=\sum\nolimits_{\{u,v\}\in E(G)}d_ud_v$, respectively. The first and second variable Zagreb indices were introduced in \cite{ref19,ref20}, which are $M_1^{(r)}(G)=\sum\nolimits_{v\in V(G)}d_v^r=\sum\nolimits_{\{u,v\}\in E(G)}\big(d_u^{r-1}+d_v^{r-1}\big)$ and $M_2^{(r)}(G)=\sum\nolimits_{\{u,v\}\in E(G)}\big(d_ud_v\big)^{r}$ (where $r$ is a variable parameter), respectively. And the generalized Zagreb index $M_{\{r,s\}}(G)=\sum\nolimits_{\{u,v\}\in E(G)}\big(d_u^rd_v^s+d_u^sd_v^r\big)$ (where $r$ and $s$ are variable parameters) was introduced in \cite{ref21}. Then the expressions of Laplacian spectral moments of power hypergraphs given in Theorem 3.2 can be represented by the Zagreb indices of graphs.

\begin{remark}
\textnormal{Let $G$ be a graph with $n$ vertices and $m$ edges. Let $d_i$ denote the degree of vertex $i$ in $G$ ($i=1,\ldots,n$). Then}
\begin{align*}
\textnormal{Tr}_d(\mathcal{L}_{G^{(k)}})&=(k-1)^{N-k}\big((k-1)^{k-1}+(-1)^kdk\big)(k-2)m+(k-1)^{N-1}M_1^{(d)}(G)\\&+(-1)^kdk^{k-2}(k-1)^{N-k}\Bigg(\sum\limits_{r=2}^{d-k+1}M_1^{(r)}(G)+\sum\limits_{r=1}^{\lfloor \frac{d-k}{2} \rfloor}M_2^{(r)}(G)+\sum\limits_{r=1}^{\lfloor \frac{d-k}{2} \rfloor}\sum\limits_{s=r+1}^{d-k-r}M_{\{r,s\}}(G)\Bigg),
\end{align*}
\textnormal{for $d=k+1,\ldots,2k-1$, and}
\begin{align*}
\textnormal{Tr}_{2k}(\mathcal{L}_{G^{(k)}})&=\ell m+(k-1)^{N-1}M_1^{(2k)}(G)+k^{2k-3}(k-1)^{N-2k+1}M_1(G)\\&+(-1)^k2k^{k-1}(k-1)^{N-k}\Bigg(\sum\limits_{r=2}^{k+1}M_1^{(r)}(G)+\sum\limits_{r=1}^{\lfloor \frac{k}{2} \rfloor}M_2^{(r)}(G)+\sum\limits_{r=1}^{\lfloor \frac{k}{2} \rfloor}\sum\limits_{s=r+1}^{k-r}M_{\{r,s\}}(G)\Bigg),
\end{align*}
\textnormal{where} $N=n+m(k-2)$ and $\ell=(k-1)^{N-k}\big((k-1)^{k-1}(k-2)+(-1)^k2k^{k-1}(k-2)+k^{k-1}-2k^{2k-3}(k-1)^{1-k}\big)$.
\end{remark}

\begin{proof}
For the terms related to degree of vertex, we have

\begin{align*}
&\sum\limits_{i=1}^{N}d_{i}^{d-k+1}+\sum\limits_{\{i,j\}\in E(G)}\sum\limits_{\begin{subarray}{c} 1\le c_i+c_j\le d-k \\ 0\le c_i,c_j<d-k \end{subarray}}d_{i}^{c_i}d_{j}^{c_j}\\
=&\sum\limits_{\{i,j\}\in E(G)}\bigg(d_i^{d-k}+d_j^{d-k}+\sum\limits_{\begin{subarray}{c} 1\le c_i+c_j\le d-k \\ 0\le c_i,c_j<d-k \end{subarray}}d_{i}^{c_i}d_{j}^{c_j}\bigg)\\
=&\sum\limits_{\{i,j\}\in E(G)}\bigg(\sum\limits_{r=1}^{d-k}\big(d_i^r+d_j^r\big)+\sum\limits_{r=1}^{\lfloor \frac{d-k}{2} \rfloor}(d_id_j)^r+\sum\limits_{r=1}^{\lfloor \frac{d-k}{2} \rfloor}\sum\limits_{s=r+1}^{d-k-r}\big(d_i^rd_j^s+d_i^sd_j^r\big)\bigg)\\
=&\sum\limits_{r=1}^{d-k}\sum\limits_{\{i,j\}\in E(G)}\big(d_i^r+d_j^r\big)+\sum\limits_{r=1}^{\lfloor \frac{d-k}{2} \rfloor}\sum\limits_{\{i,j\}\in E(G)}(d_id_j)^r+\sum\limits_{r=1}^{\lfloor \frac{d-k}{2} \rfloor}\sum\limits_{s=r+1}^{d-k-r}\sum\limits_{\{i,j\}\in E(G)}\big(d_i^rd_j^s+d_i^sd_j^r\big)\\
=&\sum\limits_{r=2}^{d-k+1}M_1^{(r)}(G)+\sum\limits_{r=1}^{\lfloor \frac{d-k}{2} \rfloor}M_2^{(r)}(G)+\sum\limits_{r=1}^{\lfloor \frac{d-k}{2} \rfloor}\sum\limits_{s=r+1}^{d-k-r}M_{\{r,s\}}(G),\ \mathrm{for}\ d=k+1,\ldots,2k.
\end{align*}

Then the expressions shown in Theorem 3.2 can be represented by the Zagreb indices of graphs.
\end{proof}

Given a $k$-uniform hypergraph $\mathcal{H}$, the \textit{signless Laplacian tensor} of $\mathcal{H}$ is $\mathcal{Q}_{\mathcal{H}}=\mathcal{D}_{\mathcal{H}}+\mathcal{A}_{\mathcal{H}}$. And the $d$-th order signless Laplacian spectral moment of $\mathcal{H}$ is equal to the $d$-th order trace of $\mathcal{Q}_{\mathcal{H}}$. For the signless Laplacian spectral moments of hypergraphs, similar conclusions as Theorem 3.1 and Theorem 3.2 can be obtained by the same method, which is shown as follows.

\begin{thm}
		\textnormal{Let $\mathcal{H}$ be a $k$-uniform hypergraph with $n$ vertices. And the degree sequence of $\mathcal{H}$ is $d_1,d_2,\ldots,d_n$. Then}
		\begin{align*} &\textnormal{Tr}_{d}(\mathcal{Q}_{\mathcal{H}})=(k-1)^{n-1}\sum\limits_{i=1}^{n}d_{i}^{d}+\textnormal{Tr}_{d}(\mathcal{A}_{\mathcal{H}})+d(k-1)^{n}\sum\limits_{z=1}^{d-1}\sum\limits_{H\in\mathcal{V}_{z}(\mathcal{H})}\sum\limits_{f\in\mathcal{F}_{d}(H)}\frac{\tau(f)\pi_{f}(\mathcal{Q}_{\mathcal{H}})}{\prod\limits_{v\in V(f)}d^{+}(v)}.
		\end{align*}
	\end{thm}

	\begin{thm}
		\textnormal{Let $G$ be a graph with $n$ vertices and $m$ edges. Let $d_i$ denote the degree of vertex $i$ in $G$ ($i=1,\ldots,n$). For the $k$-power hypergraph $G^{(k)}$ of $G$, then}
		\begin{align*}
\textnormal{Tr}_d(\mathcal{Q}_{G^{(k)}})&=(k-1)^{N-1}\sum\limits_{i=1}^nd_i^d+dk^{k-2}(k-1)^{N-k}\Big(\sum\limits_{i-1}^nd_i^{d-k+1}+\sum\limits_{\{i,j\}\in E(G)}N_{d-k}(d_i,d_j)\Big)\\
&+(k-1)^{N-k}\big((k-1)^{k-1}+dk\big)(k-2)m,\ \mathrm{for}\ d=k+1,\ldots,2k-1,\\
\textnormal{Tr}_{2k}(\mathcal{Q}_{G^{(k)}})&=(k-1)^{N-1}\sum\limits_{i=1}^nd_i^{2k}+2k^{k-1}(k-1)^{N-k}\Big(\sum\limits_{i-1}^nd_i^{k+1}+\sum\limits_{\{i,j\}\in E(G)}N_{k}(d_i,d_j)\Big)\\
&+k^{2k-3}(k-1)^{N-2k+1}\sum\limits_{i=1}^nd_i^2+qm,
		\end{align*}
		\textnormal{where} $N=n+m(k-2)$, $N_{s}(d_i,d_j)=\sum\nolimits_{\begin{subarray}{c} 1\le c_i+c_j\le s \\ 0\le c_i,c_j<s \end{subarray}}d_i^{c_i}d_j^{c_j}\big(s=1,\ldots,k\big)$, $q=(k-1)^{N-k}\big((k-1)^{k-1}(k-2)+2k^{k-1}(k-2)+k^{k-1}-2k^{2k-3}(k-1)^{1-k}\big)$.
	\end{thm}

And the signless Laplacian spectral moments of $k$-power hypergraph $G^{(k)}$ can also be represented by Zagreb indices of $G$.

Next, we introduce some concepts for the high-order (signless) Laplacian spectrum of graphs. For a graph $G$ and an integer $k\ge2$, the (signless) Laplacian spectrum of $G^{(k)}$ is called the \textit{$k$-th order (signless) Laplacian spectrum} of $G$. A graph $G$ is said to be determined by its high-order (signless) Laplacian spectrum, if there does not exist other non-isomorphic graph $H$ such that $H$ has the same $k$-th order (signless) Laplacian spectrum as $G$ for all $k\ge2$. And we give the following examples to show that some (signless) Laplacian cospectral graphs can be determined by their high-oeder (signless) Laplacian spectrum.

\begin{figure}[htbp]
		\centering
		\includegraphics[scale=0.6]{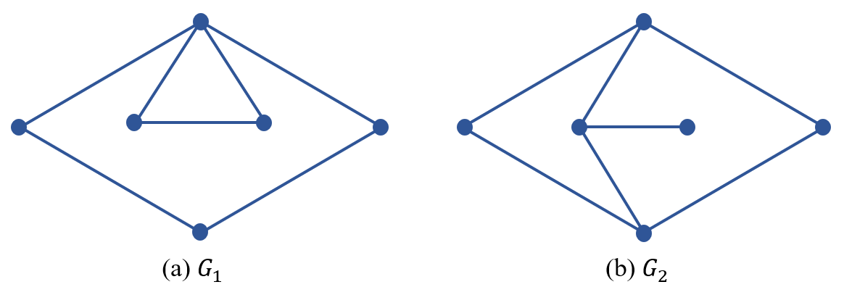}
		\caption{non-isomorphic Laplacian cospectral graph}
		\label{fig:1}
\end{figure}
	
\begin{figure}[htbp]
		\centering
		\includegraphics[scale=0.6]{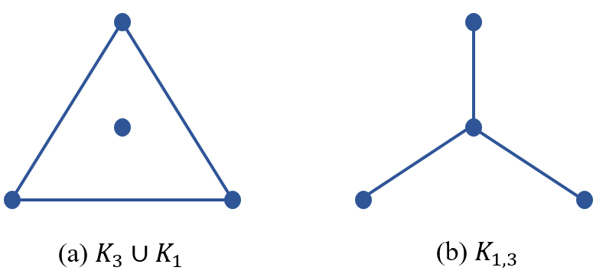}
		\caption{non-isomorphic signless Laplacian cospectral graph}
		\label{fig:2}
\end{figure}

\begin{remark}
\textnormal{The graphs shown in Figure 1 are non-isomorphic Laplacian cospectral graph. By the $3$-th order Laplacian spectral moments of $3$-power hypergraphs, we have $\textnormal{Tr}_{3}(\mathcal{L}_{(G_1)^{(3)}})\ne\textnormal{Tr}_{3}(\mathcal{L}_{(G_2)^{(3)}})$, then $(G_1)^{(3)}$ and $(G_2)^{(3)}$ have different Laplacian spectrum. So $G_1$ and $G_2$ can be distinguished by their high-order Laplacian spectrum.}

\textnormal{The graphs shown in Figure 1 are non-isomorphic signless Laplacian cospectral graph. By the $3$-th order signless Laplacian spectral moments of $3$-power hypergraphs, we have $\textnormal{Tr}_{3}(\mathcal{Q}_{(K_3\cup K_1)^{(3)}})\ne\textnormal{Tr}_{3}(\mathcal{Q}_{(K_{1,3})^{(3)}})$, then $(K_3\cup K_1)^{(3)}$ and $(K_{1,3})^{(3)}$ have different signless Laplacian spectrum. So $K_3\cup K_1$ and $K_{1,3}$ can be distinguished by their high-order signless Laplacian spectrum.}
\end{remark}

\vspace{3mm}

\noindent
\textbf{Acknowledgements}
\vspace{3mm}
\noindent

%The authors would like to thank the reviewers for giving valuable suggestions.
This work is supported by the National Natural Science Foundation of China (No.
11801115, No. 12071097, No. 12042103, No. 12242105 and No. 12371344), the Natural Science Foundation of the
Heilongjiang Province (No. QC2018002) and the Fundamental Research Funds for
the Central Universities.

\section*{References}
\bibliographystyle{unsrt}
\bibliography{spbib}
\end{spacing}
\end{document}